\documentclass[11pt]{article}

\usepackage[utf8]{inputenc}
\usepackage[T1]{fontenc}
\usepackage[a4paper,margin=1in]{geometry}
\usepackage{amsmath,amssymb,amsthm,mathtools}
\usepackage{hyperref}
\usepackage{footmisc}

\newtheorem{theorem}{Theorem}[section]
\newtheorem{proposition}[theorem]{Proposition}
\newtheorem{lemma}[theorem]{Lemma}

\theoremstyle{definition}

\theoremstyle{remark}
\newtheorem{remark}[theorem]{Remark}
\newtheorem{conjecture}[theorem]{Conjecture}

\newcommand{\R}{\mathbb{R}}

\newcommand{\E}{\mathbf{E}}
\newcommand{\Pp}{\mathbf{P}}
\newcommand{\calF}{\mathcal{F}}
\newcommand{\calN}{\mathcal{N}}
\newcommand{\calB}{\mathcal{B}}

\newcommand{\dd}{\,\mathrm{d}}

\title{The Extremal Process of a Cascading Family of\\
Branching Brownian Motions}

\author{Mohamed Ali Belloum\thanks{CEA Saclay, DES/ISAS, 91191 Gif-sur-Yvette, France.
Email: \href{mailto:mohamed-ali.belloum@cea.fr}{mohamed-ali.belloum@cea.fr}}}

\date{}

\begin{document}

\maketitle

\begin{abstract}
We study the extremal process of a critical two-type cascading
branching Brownian motion. Its critical centering is
$\sqrt2t-(2\sqrt2)^{-1}\log t$. The limiting type-2 extremal process
is a standard BBM-decorated Cox process with intensity proportional to
the type-1 derivative-martingale limit. The analysis combines the critical derivative Gibbs-measure limit,
sharp F-KPP asymptotics and killed-spine estimates to identify the
immigration contribution and its normalization. The centered maximum converges
to the corresponding derivative-martingale-shifted Gumbel law.
\end{abstract}



\bigskip

\section{Introduction}
\label{sec:introduction}

Branching Brownian motion (BBM) is one of the canonical models for the study
of traveling waves, critical martingales and extremal processes. Starting from
a single particle, each particle performs Brownian motion and branches
independently. If $\mathcal N_t$ denotes the population at time $t$ and
$X_u(t)$ the position of $u\in\mathcal N_t$, then the maximum is
\[
M_t:=\max_{u\in\mathcal N_t}X_u(t).
\]
We adopt the standard deterministic centering
\begin{equation}
 m_t:=\sqrt{2}t-\frac{3}{2\sqrt{2}}\log t.
\label{eq:front-standard}
\end{equation}
The linear speed is determined by the critical traveling wave of the
Fisher--Kolmogorov--Petrovsky--Piscounov equation \cite{KPP37}, while the
logarithmic correction is the Bramson shift established in
\cite{Bramson78,Bramson83}. The fundamental link between BBM and the F-KPP
equation was discovered independently by McKean \cite{McKean75},
Ikeda--Nagasawa--Watanabe \cite{INW69}, and Skorokhod \cite{Skorokhod64}.

The maximum and the full extremal process are governed by the derivative
martingale. In the standard normalization,
\begin{equation}
 Z_t:=\sum_{u\in\mathcal N_t}
\big(\sqrt{2}t-X_u(t)\big)
\exp\big\{\sqrt{2}X_u(t)-2t\big\}
\label{eq:derivative}
\end{equation}
converges almost surely, after the usual truncation procedure, to a
non-negative limit $Z_\infty$. The limiting law of the centered maximum has
the form
\begin{equation}
\lim_{t\to\infty}\Pp(M_t-m_t\leq x)
=\E\left[\exp\{-C^*Z_\infty e^{-\sqrt{2}x}\}\right],
\label{eq:maximum-limit}
\end{equation}
for a constant $C^*>0$ depending on the normalization. This representation was
obtained by Lalley and Sellke \cite{LalleySellke87}.

The full extremal point process is
\[
\mathcal E_t:=\sum_{u\in\mathcal N_t}\delta_{X_u(t)-m_t}.
\]
Independently, A\"id\'ekon, Berestycki, Brunet and Shi \cite{ABBS13}, and
Arguin, Bovier and Kistler \cite{ABK12,ABK13} proved that $\mathcal E_t$
converges to a decorated Poisson point process. Conditionally on $Z_\infty$,
the cluster centers form a Poisson point process with intensity proportional
to $Z_\infty e^{-\sqrt{2}x}\dd x$, while each center carries an independent
copy of a decoration supported on the non-positive half-line.

The purpose of this article is to study a critical two-type cascading model.
Type-$1$ particles form the source population. They move as standard Brownian
motions, branch at rate one into two type-$1$ particles, and, independently,
produce type-$2$ particles at rate $\alpha>0$. Type-$2$ particles perform
standard BBM and do not produce type-$1$ descendants. We write
$\mathcal N_t^1$ and $\mathcal N_t^2$ for the two populations, and $T(u)$ for
the birth time of the oldest type-$2$ ancestor of $u\in\mathcal N_t^2$.

The main result identifies the extremal process of type-$2$ particles. Its
front is shifted by a different logarithmic correction:
\begin{equation}
 m_t^{(2)}=\sqrt{2}t-\frac{1}{2\sqrt{2}}\log t.
\label{eq:type2-front}
\end{equation}
The change from $-3/(2\sqrt{2})$ to $-1/(2\sqrt{2})$ reflects the integration
over immigration times. The terminal type-$2$ family still has the standard
BBM front, but the source immigration contributes an additional factor of
order $t$.

Our contribution is twofold. First, we prove the behavior conjectured in the
critical regime of the phase-diagram analysis of reducible multitype BBM
initiated by Belloum and Mallein \cite{BelloumMallein}. More precisely, for a
cascading family with $k$ types, it is natural to conjecture that the type-$k$
extremal process is centered at
\begin{equation}
\sqrt{2}t-\frac{5-2k}{2\sqrt{2}}\log t,
\label{eq:k-type-conjecture}
\end{equation}
and that the decoration is the same standard BBM decoration $\mathcal D$,
independently of $k$, while only the intensity constant depends on the cascade
parameters. This conjecture is stated formally below as
Conjecture \ref{conj:k-type}, and it generalizes the case $k=1$ (standard BBM,
for which \eqref{eq:k-type-conjecture} reduces to
\eqref{eq:front-standard}). The present paper proves this conjecture in the
case $k=2$; the value $-1/(2\sqrt{2})$ in \eqref{eq:type2-front} is exactly the
$k=2$ case of \eqref{eq:k-type-conjecture}. Recently, a PDE proof of the front
location \eqref{eq:k-type-conjecture} for all $k$ has been obtained in
\cite{Stavrianidi2026}. Our result goes beyond the front location: it
identifies the full extremal point process as a decorated Poisson point
process, thereby describing the entire spatial structure of the type-$2$
population. Second, we make explicit the mechanism behind the limit: relevant
immigrants are born at macroscopic times and at diffusive distances below the
critical line; their independent terminal families produce the decoration,
whereas the type-$1$ derivative martingale produces the random intensity.

\subsection{Relation with previous work}

The connection between BBM and the F-KPP equation goes back to Kolmogorov,
Petrovsky and Piscounov \cite{KPP37} and to McKean \cite{McKean75}. Bramson
established the logarithmic correction in \eqref{eq:front-standard} in
\cite{Bramson78} and the convergence to travelling waves in \cite{Bramson83}.
Lalley and Sellke \cite{LalleySellke87} identified the derivative-martingale
mixture in \eqref{eq:maximum-limit}. The convergence of the full extremal
process to a decorated Poisson point process was obtained independently by
A\"id\'ekon--Berestycki--Brunet--Shi \cite{ABBS13} and by
Arguin--Bovier--Kistler \cite{ABK12,ABK13}. Their work also characterizes the
decoration as the limiting cluster seen from an unusually high particle.

For branching random walks, analogous results were proved by A\"id\'ekon
\cite{Aidekon12} and Madaule \cite{Madaule17}, while the genealogy of extremal
particles was studied in detail by Madaule and subsequent authors. Decorated
Poisson structures also arise in variable-speed BBM, multitype spatial
branching systems and other log-correlated models.

A closely related recent development is the work of Berestycki, Kim, Lubetzky,
Mallein and Zeitouni \cite{BKLMZ24} on the extremal point process of branching
Brownian motion in $\R^d$. In that setting, the limit is again a randomly
shifted decorated Poisson point process, whose clan-leaders form a Cox process
with intensity proportional to $D_\infty(\theta)e^{-\sqrt{2}r}$, and whose
decoration is the standard one-dimensional BBM decoration. Although our model
is one-dimensional and of cascading reducible type, the structural mechanism is
analogous: a random intensity produced by a derivative martingale or a
derivative-martingale-type object, multiplied by a universal decoration coming
from an independent terminal family. Our result shows that this structure
survives the passage from a single-type population to a critical two-type
cascade.

The present model is a reducible multitype BBM. Biggins \cite{Biggins95,Biggins12}
derived formulas for the speed of reducible multitype branching random walks,
while later work on multitype BBM established the phase diagram for the maximal
displacement \cite{BelloumMallein}. For the irreducible case, Ren and Yang
\cite{RenYang14} showed that the maximal displacement behaves as in the
classical BBM; the corresponding extremal process was subsequently studied by
Hou, Ren and Song \cite{HouRenSong24}, who proved convergence to a cluster
point process. For the reducible two-type case, a double jump phenomenon in
the maximum has been investigated by Ma and Ren \cite{MaRen23a}, and the
intermediate phases between normal and anomalous spreading have been analysed
in \cite{MaRen23b}. The critical line considered here is the boundary at which
the two competing mechanisms have the same linear speed. At that boundary, the
immigration integral changes the logarithmic correction and produces the front
in \eqref{eq:type2-front}. The main theorem below describes the corresponding
extremal process in the critical regime.

\subsection{Main theorem}

Define
\begin{equation}
\widehat{\mathcal E}_t
:=\sum_{u\in\mathcal N_t^2}
\delta_{X_u(t)-m_t^{(2)}}.
\label{eq:type2-process}
\end{equation}

Let $\mathcal D$ denote the standard BBM decoration: it is the limiting point
measure of the extremal cluster seen from its leader. Let $(\xi_i)_{i\geq1}$
be a Poisson point process, conditionally on $Z_\infty$, with intensity
\begin{equation}
\alpha C^*\sqrt{2}Z_\infty e^{-\sqrt{2}x}\dd x.
\label{eq:limiting-intensity}
\end{equation}
Let $(\mathcal D_i)_{i\geq1}$ be i.i.d.\ copies of $\mathcal D$, independent
of $(\xi_i)$ conditionally on $Z_\infty$, and set
\begin{equation}
\widehat{\mathcal E}_\infty
:=\sum_{i\geq1}\sum_{d\in\mathcal D_i}
\delta_{\xi_i+d}.
\label{eq:limit-process}
\end{equation}

\begin{theorem}[Extremal-process limit for the critical cascade]
\label{thm:main}
The point process $\widehat{\mathcal E}_t$ converges in distribution, for the
vague topology on locally finite point measures on $\R$, to
$\widehat{\mathcal E}_\infty$. In particular, for every non-negative continuous compactly supported
function $\phi$,
\begin{align}
\lim_{t\to\infty}
\E\left[\exp\{-\langle\widehat{\mathcal E}_t,\phi\rangle\}\right]
&=
\E\left[\exp\left\{-\alpha C^*\sqrt{2}Z_\infty
\int_{\R}e^{-\sqrt{2}z}
\left(1-\E\left[e^{-\langle\mathcal D,\phi(\cdot+z)\rangle}\right]\right)\dd z\right\} \right].
\label{eq:main-laplace}
\end{align}
Moreover,
\begin{equation}
\lim_{t\to\infty}
\Pp\left(\max_{u\in\mathcal N_t^2}X_u(t)-m_t^{(2)}\leq x\right)
=\E\left[\exp\{-\alpha C^*Z_\infty e^{-\sqrt{2}x}\}\right].
\label{eq:type2-maximum}
\end{equation}
\end{theorem}

\begin{conjecture}[$k$-type cascading family]
\label{conj:k-type}
Let $k\geq2$ and consider the $k$-type cascading BBM, where particles of type
$i\in\{1,\dots,k-1\}$ branch at rate $1$ into two particles of type $i$ and, at
rate $\alpha$, into one particle of type $i$ and one of type $i+1$, while
particles of type $k$ branch only into two particles of type $k$. Set
\[
m_t^{(k)}=\sqrt2 t-\frac{5-2k}{2\sqrt2}\log t,
\qquad
\widehat{\mathcal E}_t^{(k)}=\sum_{u\in\mathcal N_t^{(k)}}\delta_{X_u(t)-m_t^{(k)}}.
\]
Then $\widehat{\mathcal E}_t^{(k)}$ converges vaguely to a decorated Poisson
point process with the same decoration $\mathcal D$ as in Theorem \ref{thm:main}
and intensity
\[
C^*\frac{\alpha^{k-1}}{(k-1)!}\sqrt2 Z_\infty e^{-\sqrt2 x}\dd x.
\]
\end{conjecture}

\begin{remark}
Conjecture \ref{conj:k-type} is confirmed at the level of the front location
(i.e.\ for the marginal law of the maximum) by the PDE approach of
\cite{Stavrianidi2026}, for all $k$. The present paper establishes the full
point-process convergence for $k=2$, which is the first case beyond standard
BBM.
\end{remark}

\subsection{Organization of the paper}

The rest of the paper is organised as follows. In Section \ref{sec:model}, we
introduce the model and the notation used throughout. Section
\ref{sec:many-to-one} records the multitype many-to-one formula and the
Poissonian immigration identity, which are the two main probabilistic tools. In
Section \ref{sec:estimates}, we collect the standard Brownian and BBM estimates
needed for the localization argument. Section \ref{sec:localization} is devoted
to the localization of relevant immigrants, both in time
(Proposition \ref{prop:time-localization}) and in height
(Proposition \ref{prop:height-localization}). Section \ref{sec:proof-main}
contains the proof of the main theorem, combining the localization propositions
with the convergence of the conditional Laplace functional. Finally,
Section \ref{sec:discussion} discusses the result, its relation to the general
$k$-type conjecture, and possible extensions.

\section{Model and preliminary notation}
\label{sec:model}

Type-$1$ particles move as standard Brownian motions. Each type-$1$ particle
branches at rate one into two type-$1$ particles. Independently, at rate
$\alpha$, it gives birth to one type-$2$ particle at its current position.
Type-$2$ particles move as standard Brownian motions and branch at rate one
into two type-$2$ particles. They do not produce type-$1$ particles.

We write $\mathcal N_t^i$ for the set of particles of type $i$ alive at time
$t$. For $u\in\mathcal N_t^2$, let $T(u)$ be the time at which the oldest
type-$2$ ancestor on the lineage of $u$ was born from type $1$. If $u$ is a
type-$2$ immigrant born at time $s$, the descendants of $u$ during the
remaining time $t-s$ form an independent standard BBM.

For type-$1$ particles define
\begin{equation}
Z_t=\sum_{u\in\mathcal N_t^1}
\big(\sqrt{2}t-X_u(t)\big)e^{\sqrt{2}X_u(t)-2t}.
\label{eq:source-Z}
\end{equation}
The type-$1$ population is a standard BBM, so the usual truncated
derivative-martingale theory applies \cite{LalleySellke87}.

Unless specified otherwise, Laplace tests belong to $C_c^+(\mathbb R)$,
the nonnegative continuous functions with compact support.

For the auxiliary FKPP calculation we use a separate, independent standard
branching Brownian motion $\mathcal N_r^{\mathrm{BBM}}$, started from one
particle at the origin. It is not the type--2 population $\mathcal N_r^2$
of the full cascade.

\section{Many-to-one formulas}
\label{sec:many-to-one}

We first record the two identities used throughout the proof.

\begin{lemma}[Many-to-one]
\label{lem:many-to-one}
Let $f$ be a non-negative measurable functional of a Brownian path up to time
$t$. Then
\begin{equation}
\E\left[\sum_{u\in\mathcal N_t^1}f((X_u(s))_{0\leq s\leq t})\right]
=e^t\E\left[f((B_s)_{0\leq s\leq t})\right],
\label{eq:many-to-one}
\end{equation}
where $B$ is standard Brownian motion.
\end{lemma}

\begin{remark}
The many-to-one lemma in this form goes back to Kahane and Peyri\`ere
\cite{KahanePeyriere76}. See also Biggins and Kyprianou \cite{BigginsKyprianou05}
for a modern presentation.
\end{remark}

\begin{proposition}[Poissonian immigration identity]
\label{prop:immigration}
Let $f$ be a non-negative measurable functional of an immigrant born from a
type-$1$ particle. Then
\begin{align}
\E\left[\sum_{u\in\mathcal B}f(X_u,T(u))\right]
&=\alpha\int_0^\infty e^s\E\left[f(B_s,s)\right]\dd s,
\label{eq:immigration-first-moment}\\
\E\left[\exp\left\{-\sum_{u\in\mathcal B}f(X_u,T(u))\right\}\right]
&=\E\left[\exp\left\{-\alpha\int_0^\infty
\sum_{v\in\mathcal N_s^1}
\left(1-e^{-f(X_v(s),s)}\right)\dd s\right\}\right].
\label{eq:immigration-laplace}
\end{align}
Here $\mathcal B$ is the set of type-$2$ immigrants, each immigrant being
counted once at its birth time.
\end{proposition}

\begin{proof}
This is \cite[Corollary 4.3]{BelloumMallein}. We recall the argument for
completeness. Conditionally on the type-$1$ genealogy $\calF^1$, along the
lifetime of each type-$1$ particle, immigration events form an independent
Poisson point process of rate $\alpha$. By the superposition theorem for
Poisson point processes \cite[Chapter 3]{Kallenberg83}, the sum of these
independent processes is, conditionally on $\calF^1$, a Poisson point process
on $\R_+\times\R$ with intensity
\[
\mu_{\calF^1}(\dd s,\dd x) = \alpha\,\dd s \otimes \sum_{v\in\calN_s^1}\delta_{X_v(s)}(\dd x).
\]
The Laplace functional of a Poisson point process with intensity $\mu$ is
$\exp\{-\int(1-e^{-f})\dd\mu\}$, so applying it to $\mu = \mu_{\calF^1}$ and
taking the conditional expectation given $\calF^1$ yields
\[
\E\left[\exp\left\{-\sum_{u\in\calB}f(X_u,T(u))\right\}\Big|\calF^1\right]
=\exp\left\{-\alpha\int_0^\infty\sum_{v\in\calN_s^1}
\left(1-e^{-f(X_v(s),s)}\right)\dd s\right\}.
\]
Taking expectations gives \eqref{eq:immigration-laplace}. To obtain
\eqref{eq:immigration-first-moment}, we use the first-moment formula for
Poisson point processes and Tonelli's theorem for the nonnegative function $f$:
\[
\E\left[\sum_{u\in\calB}f(X_u,T(u))\right]
= \alpha\int_0^\infty \E\left[\sum_{v\in\calN_s^1}f(X_v(s),s)\right]\dd s.
\]
Applying the many-to-one formula \eqref{eq:many-to-one} with $f(\cdot,s)$ in
place of $f$ yields \eqref{eq:immigration-first-moment}.
\end{proof}

\section{Brownian and BBM estimates}
\label{sec:estimates}

We state the estimates needed in the localization argument. They are standard
consequences of the reflection principle, the killed Brownian transition
kernel, and the many-to-one lemma.

\begin{lemma}[Killing estimates: uniform endpoint strip]
\label{lem:killing}
Let $(B_s)_{s\ge0}$ be standard Brownian motion started at zero.
For all $t>0$ and $y>0$,
\begin{equation}
\Pp(B_s\ge-y\text{ for all }s\le t)
=2\Phi(y/\sqrt t)-1\le C\min\{1,y/\sqrt t\}.
\label{eq:killing}
\end{equation}
Moreover, for all $t>0$, $K>0$ and $y>0$,
\begin{equation}
\Pp(B_s<K\text{ for all }s\le t,\ B_t>K-y)
\le C\min\left\{1,\frac K{\sqrt t},\frac{K y^2}{t^{3/2}}\right\}.
\label{eq:killed-endpoint}
\end{equation}
The more precise bound
\begin{equation}
\Pp(B_s<K\text{ for all }s\le t,\ B_t\in[K-y,K])
\le \frac{2K}{\sqrt{2\pi}\,t^{3/2}}
\int_0^y z\exp\left\{-\frac{(K-z)^2}{2t}\right\}\,dz
\label{eq:killed-endpoint-gaussian}
\end{equation}
also holds.
\end{lemma}

\begin{proof}
The first identity is the reflection principle. For the second, the
killed transition density at $x<K$ is
\[
 p_t^K(x)=p_t(x)-p_t(2K-x),\qquad
 p_t(z)=(2\pi t)^{-1/2}e^{-z^2/(2t)}.
\]
Putting $x=K-z$, $z>0$, yields
\[
 p_t^K(K-z)=\frac{e^{-(K-z)^2/(2t)}}{\sqrt{2\pi t}}
 \left(1-e^{-2Kz/t}\right)
 \le \frac{2Kz}{\sqrt{2\pi}\,t^{3/2}}
 e^{-(K-z)^2/(2t)}.
\]
Integrating over $0<z<y$ proves \eqref{eq:killed-endpoint-gaussian}
and the bound $CKy^2/t^{3/2}$. The same probability is bounded
by the survival probability $2\Phi(K/\sqrt t)-1\le
C\min\{1,K/\sqrt t\}$, proving \eqref{eq:killed-endpoint}.
\end{proof}

\begin{lemma}[BBM upper tail]
\label{lem:BBM-tail}
There exists $C<\infty$ such that for all $r\geq1$ and $y\geq0$,
\begin{equation}
\Pp\left(M_r>\sqrt{2}r-\frac{3}{2\sqrt{2}}\log r+y\right)
\leq C(1+y)e^{-\sqrt{2}y}.
\label{eq:BBM-tail}
\end{equation}
A sharper bound with an additional Gaussian factor is available in the
diffusive regime \cite{ABK12,ABK13}.
\end{lemma}

\begin{remark}
The upper tail estimate goes back to Bramson \cite{Bramson78}, and was refined
by Arguin, Bovier and Kistler \cite[Corollary 10]{ABK12}.
\end{remark}

\begin{lemma}[Central FKPP asymptotic: precise normalization]
\label{lem:extremal-laplace-tail}
Let $\phi\in C_c^+(\R)$, put $\ell_t=(\log t)/(2\sqrt2)$ and
$u_\phi(r,z)=1-\E[\exp\{-\sum_{v\in\mathcal N_r^{\mathrm{BBM}}}
\phi(X_v(r)-z)\}]$, where the expectation is taken for an independent
standard BBM started from a single particle at the origin.  For $0<\delta<1/2$ and $0<\rho<R<\infty$,
the asymptotic needed in the proof of Theorem~\ref{thm:main} is
\begin{equation}
\label{eq:uniform-laplace-tail}
\sup_{\substack{\delta\le r/t\le1-\delta\\
                  \rho\le d/\sqrt t\le R}}
\left|\frac{u_\phi(r,\sqrt2r+d-\ell_t)}
{\sqrt t\,d\,r^{-3/2}\exp\{-\sqrt2d-d^2/(2r)\}}
-\gamma_0(\phi)\right|\longrightarrow0,
\end{equation}
where the decoration normalization is
\begin{equation}
\label{eq:gamma0}
\gamma_0(\phi)=C^*\sqrt2\int_{\R}e^{-\sqrt2z}
\left(1-\E[e^{-\langle\mathcal D,\phi(\cdot+z)\rangle}]\right)\dd z.
\end{equation}
\end{lemma}

\begin{proof}[Uniform ABK comparison, cutoff removal and identification]
We work with a single independent standard BBM.  If $\phi\equiv0$,
then $u_\phi\equiv0$ and the assertion is immediate.  In what follows
$\phi\not\equiv0$, and $\delta,\rho,R$ are fixed as in the statement.
The constants in the estimates below may depend on these parameters.

\smallskip
\noindent\emph{Step 1: a datum to which the ABK comparison applies.}
Fix $B>\max\{0,\sup\operatorname{supp}\phi\}$ and write
\[
 u_B(r,z):=1-\E\!\left[
  \exp\!\left\{-\sum_{v\in\mathcal N_r^{\mathrm{BBM}}}
                    \phi(X_v(r)-z)\right\}
     \mathbf1_{\{M_r\le z+B\}}\right].
\]
The associated F--KPP initial datum is
\[
 f_B(x)=1-e^{-\phi(-x)}\mathbf1_{\{-x\le B\}}.
\]
It equals $1$ for $x<-B$ and vanishes for $x$ sufficiently
large.  It is measurable, takes its values in $[0,1]$, and satisfies
the hypotheses of \cite[Theorem~4.2 and Proposition~4.3]{ABK13};
this is precisely the cutoff used in
\cite[Lemma~4.9, (4.48)--(4.49)]{ABK13}.

\smallskip
\noindent\emph{Step 2: uniform asymptotics of the auxiliary kernel.}
For fixed $q>0$ let
\[
 w_{B,q}(y)=u_B(q,\sqrt2q+y),\quad
 b_r=\frac{3}{2\sqrt2}\log r,\quad
 X=d-\ell_t,\quad T=r-q.
\]
In the range under consideration $X>0$ for large $t$, and
\[
 0<c_{\delta,\rho,R}\le\frac{X}{\sqrt r}
     \le C_{\delta,\rho,R}<\infty,
 \qquad \frac Xr\longrightarrow0
\]
uniformly.  With the variables of \cite[(4.7)]{ABK13}, the comparison
kernel is
\begin{align*}
 \psi_B(q,r,\sqrt2r+X)
  &=\frac{e^{-\sqrt2X}}{\sqrt{2\pi T}}
     \int_0^\infty w_{B,q}(y)e^{\sqrt2y}
         \exp\!\left(-\frac{(y-X)^2}{2T}\right)\\[-2pt]
  &\hspace{52mm}\times
      \left(1-\exp\!\left\{-\frac{2y(X+b_r)}T\right\}\right)\,\dd y.
\end{align*}
Introduce
\[
\begin{aligned}
 H_r(X)&=Xr^{-3/2}\exp\!\left(-\sqrt2X-\frac{X^2}{2r}\right),\\
 A_{q,r,X}(y)&=\frac{r^{3/2}}{X\sqrt{2\pi T}}
 e^{\frac{X^2}{2r}-\frac{(y-X)^2}{2T}}
 \left(1-e^{-2y(X+b_r)/T}\right).
\end{aligned}
\]
For each $L<\infty$, elementary Taylor estimates show that,
uniformly for $y\in[0,L]$ and for all $(r,d)$ in the indicated
central rectangle,
\[
 A_{q,r,X}(y)\longrightarrow\sqrt{\frac2\pi}\,y.
\]
Indeed, $T/r\to1$, $b_r/X\to0$, $X/T=O(r^{-1/2})$ and
\[
 \frac{X^2}{2r}-\frac{(y-X)^2}{2T}
 =-\frac{qX^2}{2rT}+\frac{yX}{T}-\frac{y^2}{2T}
 \longrightarrow0
\]
uniformly on such compact $y$-sets.  For all sufficiently large $r$
(and fixed $q$), $1-e^{-z}\le z$ for $z\ge0$ gives the
\emph{uniform} domination
\[
 0\le A_{q,r,X}(y)
 \le C_{q,\delta,\rho,R}\,y e^{y},\qquad y\ge0.
\]
Here we used $X/T\le1$, $b_r/X\le1$ and the nonpositivity of
$-qX^2/(2rT)-y^2/(2T)$.
Writing $y_0(B)=\sup\{x:f_B(x)>0\}<\infty$, the many-to-one lemma
and a Gaussian tail bound yield, for fixed $q$,
\[
 w_{B,q}(y)
 \le e^q\Pp\bigl(B_q\ge\sqrt2q+y-y_0(B)\bigr),
 \qquad y\ge0,
\]
where $B_q\sim N(0,q)$.  Consequently
\[
 \int_0^\infty y e^{(\sqrt2+1)y}w_{B,q}(y)\,\dd y<\infty.
\]
Dominated convergence, first on $[0,L]$ and then using this
integrable envelope for the tail, therefore proves
\begin{equation}\label{eq:fkpp-fixed-q-uniform}
 \sup_{\substack{\delta\le r/t\le1-\delta\\
                        \rho\le d/\sqrt t\le R}}
 \left|\frac{\psi_B(q,r,\sqrt2r+X)}{H_r(X)}-C_B(q)\right|
 \longrightarrow0,
 \quad
 C_B(q)=\sqrt{\frac2\pi}\int_0^\infty
       y e^{\sqrt2y}w_{B,q}(y)\,\dd y.
\end{equation}

\smallskip
\noindent\emph{Step 3: removal of the auxiliary time.}
For all sufficiently large fixed $q$, \cite[Proposition~4.3,
(4.8)]{ABK13} supplies a number $\eta_B(q)\ge1$ satisfying
$\eta_B(q)\downarrow1$ as $q\to\infty$, such that
\[
 \eta_B(q)^{-1}\psi_B(q,r,\sqrt2r+X)
 \le u_B(r,\sqrt2r+X)
 \le\eta_B(q)\psi_B(q,r,\sqrt2r+X).
\]
This applies uniformly in the rectangle for large $t$, because
$r\ge8q$ and $X\ge8q-b_r$ there.  ABK Lemma~4.9 guarantees
\[
 \lim_{q\to\infty}C_B(q)=C_B(\phi)\in(0,\infty).
\]
Combining the two-sided inequality with
\eqref{eq:fkpp-fixed-q-uniform} and taking $t\to\infty$
\emph{before} $q\to\infty$ yields
\begin{equation}\label{eq:fkpp-uniform-cutoff}
 \sup_{\substack{\delta\le r/t\le1-\delta\\
                        \rho\le d/\sqrt t\le R}}
 \left|\frac{u_B(r,\sqrt2r+X)}{H_r(X)}-C_B(\phi)\right|
 \longrightarrow0.
\end{equation}

\smallskip
\noindent\emph{Step 4: the logarithmic shift and the normalization.}
For
$K_{t,r}(d)=\sqrt t\,d\,r^{-3/2}
\exp(-\sqrt2d-d^2/(2r))$, the identity
$e^{\sqrt2\ell_t}=\sqrt t$ gives
\[
 \frac{H_r(d-\ell_t)}{K_{t,r}(d)}
 =\left(1-\frac{\ell_t}{d}\right)
    \exp\!\left(\frac{2d\ell_t-\ell_t^2}{2r}\right)
 \longrightarrow1
\]
uniformly in the rectangle, since $d\ge\rho\sqrt t$ and
$r\ge\delta t$.  Thus \eqref{eq:fkpp-uniform-cutoff}
holds with $H_r(X)$ replaced by $K_{t,r}(d)$.

\smallskip
\noindent\emph{Step 5: uniform removal of the cutoff.}
For $B'>B$, the monotonicity and union bound give, for every $r,z$,
\[
 0\le u_B(r,z)-u_{B'}(r,z)
 \le u_B(r,z)-u_\phi(r,z)
 \le\Pp(M_r>z+B).
\]
Apply Steps~1--4 also to the Heaviside initial datum
$f(x)=\mathbf1_{\{x<0\}}$, using the standard BBM maximum
constant $C^*$ with front
$\sqrt2r-3(2\sqrt2)^{-1}\log r$.  The same fixed-$q$
comparison and \cite[Lemma~4.5]{ABK13} imply, for each fixed $B$,
\[
 \sup_{\substack{\delta\le r/t\le1-\delta\\
                        \rho\le d/\sqrt t\le R}}
 \left|\frac{\Pp(M_r>\sqrt2r+d-\ell_t+B)}{K_{t,r}(d)}
             -C^*e^{-\sqrt2B}\right|\longrightarrow0.
\]
Consequently, $B\mapsto C_B(\phi)$ decreases to a finite
$C(\phi)\ge0$, and
\[
 0\le C_B(\phi)-C(\phi)\le C^*e^{-\sqrt2 B}.
\]
For each $B$, the triangle inequality yields
\begin{align*}
 \sup_{\mathrm{central}}\left|
   \frac{u_\phi(r,\sqrt2r+d-\ell_t)}{K_{t,r}(d)}-C(\phi)
 \right|
 &\le
 \sup_{\mathrm{central}}\left|
   \frac{u_B(r,\sqrt2r+d-\ell_t)}{K_{t,r}(d)}-C_B(\phi)
 \right|\\
 &\quad+\sup_{\mathrm{central}}
 \frac{\Pp(M_r>\sqrt2r+d-\ell_t+B)}{K_{t,r}(d)}
       +C_B(\phi)-C(\phi).
\end{align*}
Taking $t\to\infty$ first and then $B\to\infty$ proves
\eqref{eq:uniform-laplace-tail} with $C(\phi)$ in place of
$\gamma_0(\phi)$.  By \cite[Lemma~4.10, (4.55)--(4.56)]{ABK13},
this $C(\phi)$ is exactly the coefficient in the limiting Laplace
functional of the one-type BBM at the \emph{standard} centering
$m_r=\sqrt2r-3(2\sqrt2)^{-1}\log r$.

\smallskip
\noindent\emph{Step 6: the decorated constant.}
At translation zero, \cite[Proposition~3.2]{ABK13} gives
\[
 \lim_{r\to\infty}\E\!\left[e^{-\langle\mathcal E_r,\phi\rangle}\right]
 =\E\left[e^{-C(\phi)Z_\infty}\right].
\]
On the other hand, \cite[Theorem~2.1]{ABK13} describes the same
limit as a Poisson point process of centers having conditional
intensity $C^*\sqrt2 Z_\infty e^{-\sqrt2z}\,\dd z$, each independently
marked by a copy of $\mathcal D$.  Its Poisson Laplace formula is
\[
 \E\exp\!\left\{-C^*\sqrt2 Z_\infty
   \int_\R e^{-\sqrt2z}
   \left(1-\E[e^{-\langle\mathcal D,\phi(\cdot+z)\rangle}]\right)
       \dd z\right\}
 =\E[e^{-\gamma_0(\phi)Z_\infty}].
\]
The mark integral is finite because it is the Laplace exponent of
the limiting locally finite decorated process for $\phi\in C_c^+(\R)$.
As $Z_\infty>0$ almost surely for the binary one-type BBM,
$c\mapsto\E[e^{-cZ_\infty}]$ is strictly decreasing for $c\ge0$.
It follows that $C(\phi)=\gamma_0(\phi)$.
\end{proof}

\begin{lemma}[Translation of the central maximum kernel]
\label{lem:max-translation}
For $d>0$ put
\[
K_{t,r}(d)=\sqrt t\,d\,r^{-3/2}\exp\{-\sqrt2d-d^2/(2r)\}.
\]
Fix $x\in\mathbb R$, $0<\delta<1/2$, and $0<\rho<R<\infty$.
Then
\begin{equation}\label{eq:max-kernel-translation}
\sup_{\substack{\delta\le r/t\le1-\delta\\
\rho\le d/\sqrt t\le R}}
\left|\frac{K_{t,r}(d+x)}{K_{t,r}(d)}
-e^{-\sqrt2x}\right|\longrightarrow0.
\end{equation}
Consequently, if for some $C_*>0$
\begin{equation}\label{eq:max-zero-input}
\sup_{\substack{\delta\le r/t\le1-\delta\\
\rho\le d/\sqrt t\le R}}
\left|\frac{q_{t,r,0}(d)}{K_{t,r}(d)}-C_*\right|
\longrightarrow0
\end{equation}
for every such $\delta,\rho,R$, then for each fixed $x$
\begin{equation}\label{eq:max-x-input}
\sup_{\substack{\delta\le r/t\le1-\delta\\
\rho\le d/\sqrt t\le R}}
\left|\frac{q_{t,r,x}(d)}{K_{t,r}(d)}
-C_*e^{-\sqrt2x}\right|\longrightarrow0.
\end{equation}
\end{lemma}
\begin{proof}
For $t$ sufficiently large $d+x>0$ on the specified window, and
\[
\frac{K_{t,r}(d+x)}{K_{t,r}(d)}
=e^{-\sqrt2x}\left(1+\frac xd\right)
\exp\left\{-\frac{2dx+x^2}{2r}\right\}.
\]
Since $d\ge\rho\sqrt t$, $d\le R\sqrt t$, and $r\ge\delta t$,
the two correction factors converge uniformly to $1$.  For the
second assertion, use the exact identity
$q_{t,r,x}(d)=q_{t,r,0}(d+x)$ and apply
\eqref{eq:max-zero-input} on a slightly enlarged fixed height
window, e.g. $[\rho/2,2R]$, then use
\eqref{eq:max-kernel-translation}.
\end{proof}

\begin{lemma}[Maximum: direct conditional-Poisson reduction]
\label{lem:maximum-direct}
For $x\in\R$, define
\[
q_{t,r,x}(d)=\Pp\big(M_r\ge\sqrt2r+d-\ell_t+x\big),
\qquad
\Lambda_t^{\max}(x)=\alpha\int_0^t\sum_{v\in\mathcal N_s^1}
q_{t,t-s,x}(D_v(s))\dd s.
\]
Then, for every $t>0$,
\begin{equation}\label{eq:max-exact-poisson}
\Pp\big(\max_{u\in\mathcal N_t^2}X_u(t)-m_t^{(2)}\le x\big)
=\E\big[e^{-\Lambda_t^{\max}(x)}\big],
\end{equation}
with the usual harmless choice of strict inequalities at the threshold.
Consequently, if for each $x\in\R$,
\begin{equation}\label{eq:max-exponent-needed}
\Lambda_t^{\max}(x)\xrightarrow{\Pp}
\alpha C^* Z_\infty e^{-\sqrt2x},
\end{equation}
then the maximum convergence asserted in
\eqref{eq:type2-maximum} follows.
\end{lemma}
\begin{proof}
Condition on the type-$1$ genealogy. An immigrant at time $s$ and
position $X_v(s)$ independently produces an exceedance with
probability $q_{t,t-s,x}(D_v(s))$. By Poisson thinning,
the number of successful immigrant families is Poisson with mean
$\Lambda_t^{\max}(x)$. The event of no success
therefore has conditional probability $e^{-\Lambda_t^{\max}(x)}$.
Taking expectations proves \eqref{eq:max-exact-poisson}.  If
\eqref{eq:max-exponent-needed} holds, bounded continuity of $e^{-z}$
proves the limit.  No inference from vague convergence to convergence
of the maximum is used.
\end{proof}

\section{Localization of relevant immigrants}
\label{sec:localization}

Fix $A>0$ and $\varepsilon>0$. We show that particles contributing to the
extremal process are generated by immigrants in a macroscopic time window and
in a diffusive spatial window.

\begin{lemma}[Exclusion of subdiffusively late immigration]
\label{lem:very-late-immigration}
Let $m_t^{(2)}=\sqrt2t-(2\sqrt2)^{-1}\log t$ and let $A>0$.
For every $\beta\in(0,1/2)$,
\[
\lim_{t\to\infty}\mathbb P\bigl(\exists u\in\mathcal N_t^2:
 X_u(t)\ge m_t^{(2)}-A,\ T(u)\ge t-t^\beta\bigr)=0.
\]
\end{lemma}
\begin{proof}
Set $L_t=(2\sqrt2)^{-1}\log t+A$ and
$F_{t,r}(d)=\mathbb P(M_r\ge\sqrt2r+d-L_t)$, where $M_r$ is
 the maximum of an independent standard BBM of duration $r$.
The BBM many-to-one formula and the Gaussian identity
$e^r p_r(\sqrt2r+z)=e^{-\sqrt2z}p_r(z)$ give, for $r>0$,
\begin{align*}
 e^{\sqrt2d}F_{t,r}(d)
 &\le e^{\sqrt2d+r}\int_{\sqrt2r+d-L_t}^\infty p_r(x)\,dx\\
 &=e^{\sqrt2L_t}\int_{d-L_t}^\infty
    e^{-\sqrt2(z-d+L_t)}p_r(z)\,dz\\
 &\le e^{\sqrt2L_t}\mathbb P(B_r\ge d-L_t).
\end{align*}
The same inequality holds at $r=0$ directly.
Fix $a>0$ and discard type-1 ancestral paths crossing
$D_v(q)=\sqrt2q-X_v(q)=-a$. The probability that any such crossing
occurs is at most $e^{-\sqrt2a}$ by the additive-martingale maximal
inequality. Write $p_s^+(a,y)=p_s(y-a)-p_s(y+a)$ for $y>0$.
By Poisson thinning, the type-1 many-to-one formula and the preceding
bound, the expected number $N_{t,a}$ of successful immigrants with
$s\in[t-t^\beta,t]$ and non-crossing ancestors is at most
\[
 \alpha e^{\sqrt2L_t}\int_0^{t^\beta}\!dr\int_0^\infty
 p_{t-r}^+(a,y)\,\mathbb P(B_r\ge y-a-L_t)\,dy.
\]
For sufficiently large $t$, $t-r\ge t/2$ throughout the integral.
The reflection formula and $1-e^{-z}\le z$ imply
\[
 p_{t-r}^+(a,y)\le C_a yt^{-3/2}.
\]
By Tonelli's theorem,
\[
 \int_0^\infty y\,\mathbb P(B_r\ge y-a-L_t)\,dy
 =\frac12\mathbb E[(a+L_t+B_r)_+^2]
 \le C_a(1+L_t^2+r).
\]
Since $e^{\sqrt2L_t}=e^{\sqrt2A}\sqrt t$, we conclude
\[
 \mathbb E N_{t,a}\le C_{a,A,\alpha}
 \left[t^{\beta-1}(1+\log^2t)+t^{2\beta-1}\right]\longrightarrow0.
\]
Markov's inequality and the crossing bound yield
\[
 \limsup_{t\to\infty}\mathbb P(\text{a successful immigrant in }
 [t-t^\beta,t])\le e^{-\sqrt2a}.
\]
Letting $a\to\infty$ finishes the proof.
\end{proof}

\begin{lemma}[Exclusion of immigration in a fixed initial interval]
\label{lem:fixed-early-immigration}
Let $A>0$ and $S<\infty$ be fixed. For the critical two-type
branching Brownian motion,
\[
\lim_{t\to\infty}\mathbb P\bigl(\exists u\in\mathcal N_t^2:
 X_u(t)\ge m_t^{(2)}-A,\ T(u)\le S\bigr)=0.
\]
\end{lemma}
\begin{proof}
The number $K_S$ of type-2 immigration events before time $S$ is
finite almost surely: the type-1 Yule population has finitely many
individuals on $[0,S]$ almost surely, and the total immigration
intensity $\alpha\int_0^S|\mathcal N_s^1|\,ds$ is finite almost surely.
Enumerate these events by $(\tau_i,\xi_i)$, $1\le i\le K_S$,
where $\tau_i\le S$ and $\xi_i$ is the position of the immigrant.
Conditional on the type-1 process and immigration events, the
terminal descendants of event $i$ have maximum
$\xi_i+M_{t-\tau_i}^{(i)}$, where the $M^{(i)}$ are independent
standard branching Brownian maxima. Let
\[
 b_r=\sqrt2r-\frac{3}{2\sqrt2}\log r,\qquad r\ge1.
\]
The family $(M_r-b_r)_{r\ge1}$ is tight by the standard
Bramson maximum theorem. Moreover, uniformly for $0\le s\le S$,
\[
 m_t^{(2)}-A-\xi-b_{t-s}
 =\sqrt2s-\xi-A+
 \frac{3\log(t-s)-\log t}{2\sqrt2}
 =\frac{\log t}{\sqrt2}+O_{S,\xi,A}(1)
 \xrightarrow[t\to\infty]{}+\infty.
\]
For each fixed realization of the finite list
$(\tau_i,\xi_i)_{i\le K_S}$, tightness and a finite union bound
therefore imply that the conditional probability of any such
immigrant producing a particle above $m_t^{(2)}-A$ tends to zero.
Dominated convergence, since these conditional probabilities lie
in $[0,1]$, gives the assertion.
\end{proof}

\begin{lemma}[Gaussian-refined BBM upper tail]
There exist constants $C,c>0$ and $T<\infty$ such that, for every $q\ge T$ and $X\ge0$,
\[
 \mathbb P(M_q\ge\sqrt2 q+X)
 \le C\frac{X+\log q}{q^{3/2}}e^{-\sqrt2 X}e^{-cX^2/q}.
\]
\end{lemma}
\begin{proof}
Write $u(q,x)=\mathbb P(M_q\ge x)$. Fix a sufficiently large but finite
$q_0$ for which the upper comparison in Arguin--Bovier--Kistler
(2013), Proposition 4.3, holds. For $q\ge8q_0$ and $X\ge0$,
that comparison and $1-e^{-v}\le v$ give
\[
 u(q,\sqrt2q+X)\le
 C\frac{e^{-\sqrt2X}(X+\frac{3}{2\sqrt2}\log q)}{(q-q_0)^{3/2}}
 \int_0^\infty y e^{\sqrt2y}u(q_0,\sqrt2q_0+y)
 \exp\!\left[-\frac{(y-X)^2}{2(q-q_0)}\right]dy.
\]
The fixed-time many-to-one inequality yields
$u(q_0,\sqrt2q_0+y)\le e^{q_0}\mathbb P(B_{q_0}\ge\sqrt2q_0+y)$,
which decays quadratically in $y$. Since
$(y-X)^2\ge X^2/2-y^2$, the integral is bounded by
\[
 e^{-X^2/[4(q-q_0)]}
 \int_0^\infty y e^{\sqrt2y}u(q_0,\sqrt2q_0+y)
 e^{y^2/[2(q-q_0)]}\,dy.
\]
The last integral is bounded uniformly for $q\ge8q_0$ by the
fixed-time Gaussian decay. Since $q-q_0\asymp q$, the claim follows
(with e.g. $c=1/4$ after enlarging $T$).
\end{proof}

\begin{proposition}[Unconditional central large-height localization]
Fix $a>0$, $A\in\mathbb R$ and $\delta\in(0,1/2)$. Let
$L_t=(\log t)/(2\sqrt2)+A$, $r=t-s$,
$F_{t,r}(d)=\mathbb P(M_r\ge\sqrt2r+d-L_t)$ and
$p_s^+(a,y)=p_s(y-a)-p_s(y+a)$. Then
\[
\lim_{R\to\infty}\limsup_{t\to\infty}
\int_{\delta t}^{(1-\delta)t}\!ds
\int_{a+R\sqrt t}^\infty\!dy\,
 e^{\sqrt2(y-a)}p_s^+(a,y)F_{t,t-s}(y-a)=0.
\]
\end{proposition}
\begin{proof}
Fix $R>0$. For $d=y-a\ge R\sqrt t$ and large $t$,
$X=d-L_t\ge d/2>0$, while $r\in[\delta t,(1-\delta)t]$.
The preceding lemma gives
\[
 e^{\sqrt2d}F_{t,r}(d)
 \le C_{A,\delta}\frac{\sqrt t}{r^{3/2}}
 (d+\log t)e^{-c_\delta d^2/t}.
\]
The exact reflection formula gives
$p_s^+(a,y)\le C_a y s^{-3/2}$.
Integrating first over $s$ yields
\[
 C_{a,A,\delta}t^{-3/2}
 \int_{R\sqrt t}^\infty(d+a)(d+\log t)e^{-c_\delta d^2/t}\,dd.
\]
Setting $d=z\sqrt t$ and taking $t\to\infty$ bounds its
limit superior by
$C_{a,A,\delta}\int_R^\infty z^2e^{-c_\delta z^2}\,dz$,
which tends to zero as $R\to\infty$.
\end{proof}

\begin{proposition}[Early macroscopic immigration window]
Fix $a>0$, $A\in\mathbb R$ and $\alpha>0$. Put
$L_t=(\log t)/(2\sqrt2)+A$, $r=t-s$,
$F_{t,r}(d)=\mathbb P(M_r\ge\sqrt2r+d-L_t)$, and
$p_s^+(a,y)=p_s(y-a)-p_s(y+a)$ for $y>0$.
Assume the standard Bramson upper tail
$\mathbb P(M_r\ge b_r+h)\le C(1+h_+)e^{-\sqrt2h}$,
$r\ge1$, $h\in\mathbb R$, where
$b_r=\sqrt2r-\frac{3}{2\sqrt2}\log r$.
Then for every $0<\varepsilon<1/2$,
\[
\limsup_{t\to\infty}\alpha\int_0^{\varepsilon t}\!ds
\int_0^\infty\!dy\,e^{\sqrt2(y-a)}p_s^+(a,y)
F_{t,t-s}(y-a)\le C_{a,A,\alpha}\varepsilon.
\]
Consequently the double limit as $\varepsilon\downarrow0$ is zero.
\end{proposition}
\begin{proof}
For $0\le s\le\varepsilon t$ we have $r\in[t/2,t]$.
Writing $d=y-a\ge-a$ and
$h=d-L_t+\frac{3}{2\sqrt2}\log r$, we have
$h\ge\frac{\log t}{\sqrt2}-O_{a,A}(1)>0$
for all sufficiently large $t$, uniformly in $s$ and $y$.
The assumed BBM bound therefore gives
\[
e^{\sqrt2d}F_{t,r}(d)
\le C e^{\sqrt2L_t}r^{-3/2}(1+h)
\le \frac{C_{a,A}}t(1+y+\log t).
\]
For killed Brownian motion on $(0,\infty)$,
\[
\int_0^\infty y\,p_s^+(a,y)\,dy=a,
\qquad
\int_0^\infty p_s^+(a,y)\,dy
=\mathbb P_a(\tau_0>s)\le\min\{1,Ca/\sqrt s\}.
\]
Thus the desired integral is at most
\[
\frac{C_{a,A,\alpha}}t\int_0^{\varepsilon t}
\bigl[a+(1+\log t)\min\{1,Ca/\sqrt s\}\bigr]ds
\le C_{a,A,\alpha}\left(a\varepsilon+
\frac{(1+\log t)(1+a\sqrt{\varepsilon t})}{t}\right).
\]
Taking $t\to\infty$ proves the claim.
\end{proof}

\begin{proposition}[Late macroscopic immigration window]
Fix $a>0$, $A\in\mathbb R$ and $\alpha>0$. Let
$L_t=(\log t)/(2\sqrt2)+A$, $r=t-s$, and
$F_{t,r}(d)=\mathbb P(M_r\ge\sqrt2r+d-L_t)$.
Suppose the standard BBM tail bound, for $r\ge r_0$ and $X\ge0$,
\[
 \mathbb P(M_r\ge\sqrt2r+X)
 \le C\frac{X+\log r}{r^{3/2}}e^{-\sqrt2X}e^{-cX^2/r},
\]
and the ordinary Bramson bound for arbitrary shifts.
Then
\[
\lim_{\varepsilon\downarrow0}\limsup_{t\to\infty}
\alpha\int_{(1-\varepsilon)t}^{t}\!ds\int_0^\infty\!dy\,
 e^{\sqrt2(y-a)}p_s^+(a,y)F_{t,t-s}(y-a)=0.
\]
\end{proposition}
\begin{proof}
Fix $\beta\in(0,1/2)$ and $\varepsilon<1/2$. The interval $0<r<t^\beta$ contributes $o(1)$ to the
\emph{killed-spine first moment}, by the explicit estimate proved in
Lemma~\ref{lem:very-late-immigration}:
\[
\alpha\int_0^{t^\beta}\!dr\int_0^\infty\!dy\,
 e^{\sqrt2(y-a)}p_{t-r}^+(a,y)F_{t,r}(y-a)
 \le C_{a,A,\alpha}
 \left[t^{\beta-1}(1+\log^2 t)+t^{2\beta-1}\right]=o(1).
\]
This invokes the quantitative first-moment estimate, not merely the
probability conclusion of that lemma. It remains to consider
$t^\beta\le r\le\varepsilon t$, where $s=t-r\ge t/2$ and
\[
 p_s^+(a,y)\le C_a yt^{-3/2}.
\]
Write $d=y-a$ and $X=d-L_t$. Since $L_t=O(\log t)$ and
$r\ge t^\beta$, the Bramson shift
$h=X+(3/(2\sqrt2))\log r$ is bounded below by
$-C_{a,A}\log t$.
For $X<0$ (hence $0<y<a+L_t$), the ordinary Bramson bound
$F_{t,r}(d)\le C(1+h_+)e^{-\sqrt2h}$ implies
\[
 e^{\sqrt2d}F_{t,r}(d)
 \le C_A\frac{\sqrt t}{r^{3/2}}(1+\log t).
\]
For $X\ge0$, the Gaussian-refined bound yields
\[
 e^{\sqrt2d}F_{t,r}(d)
 \le C_A\frac{\sqrt t}{r^{3/2}}
 (X+\log r)e^{-cX^2/r}.
\]
Consequently, with $b_t=a+L_t=O(\log t)$, the inner
$y$-integral, after multiplication by $p_s^+$, is at most
\[
 \frac{C_{a,A}}{t r^{3/2}}\left[
 (1+\log t)b_t^2+
 \int_0^\infty (b_t+X)(X+\log r)e^{-cX^2/r}\,dX
 \right].
\]
Gaussian moments give, uniformly in $r\ge t^\beta$,
\[
 \int_0^\infty (b_t+X)(X+\log r)e^{-cX^2/r}\,dX
 \le C\bigl(r^{3/2}+b_t r+r\log r+b_t\sqrt r\log r\bigr).
\]
Integration over $r\in[t^\beta,\varepsilon t]$ therefore gives
\[
 C_{a,A,\alpha}\left[
 \varepsilon+
 \frac{(1+\log t)^2\sqrt\varepsilon}{\sqrt t}
 +\frac{(1+\log t)^3}{t}\right].
\]
The displayed estimate is valid for all sufficiently large $t$;
constants may depend on $a,A,\alpha,\beta$ but not on
$t,\varepsilon$. Its limit superior is $O(\varepsilon)$.
Adding the $o(1)$ very-late contribution and then letting
$\varepsilon\downarrow0$ completes the proof.
\end{proof}

\begin{lemma}[Small-height localization on a central time window]
Fix $a>0$, $A\in\mathbb R$, and $\delta\in(0,1/2)$. Write
$L_t=(\log t)/(2\sqrt2)+A$, $r=t-s$,
$F_{t,r}(d)=\mathbb P(M_r\ge\sqrt2r+d-L_t)$, and
$p_s^+(a,y)=p_s(y-a)-p_s(y+a)$. Then
\[
\lim_{\rho\downarrow0}\limsup_{t\to\infty}
\int_{\delta t}^{(1-\delta)t}\!ds
\int_{\{y>0:y-a<\rho\sqrt t\}}\!dy\,
 e^{\sqrt2(y-a)}p_s^+(a,y)F_{t,t-s}(y-a)=0.
\]
\end{lemma}
\begin{proof}
The reflection formula and $1-e^{-z}\le z$ give, for $s\ge\delta t$,
\[
 p_s^+(a,y)\le \frac{2ay}{\sqrt{2\pi}s^{3/2}}
 \le C_{a,\delta}y t^{-3/2}.
\]
Put $d=y-a$ and $h=d-L_t+(3/(2\sqrt2))\log r$.
The standard centered BBM upper tail for $h\ge0$, combined
with the trivial bound $F_{t,r}(d)\le1$ for $h<0$, implies
for every real $h$ that
\[
 F_{t,r}(d)\le C(1+h_+)e^{-\sqrt2h}.
\]
Indeed, for $h<0$ the right-hand side is at least one if $C\ge1$.
Consequently, uniformly over $r\in[\delta t,(1-\delta)t]$,
\[
 e^{\sqrt2d}F_{t,r}(d)
 \le C_A\frac{\sqrt t}{r^{3/2}}(1+h_+)
 \le \frac{C_{A,\delta}}t(1+y+\log t).
\]
Since $0<y<a+\rho\sqrt t$, the required integral is at most
\[
 C_{a,A,\delta}\,t\cdot t^{-3/2}\cdot t^{-1}
 \int_0^{a+\rho\sqrt t} y(1+y+\log t)\,dy
 \le C_{a,A,\delta}\left(\rho^3+
 \frac{\rho^2(1+\log t)}{\sqrt t}
 +\frac{(1+\log t)}{t^{3/2}}\right).
\]
Taking $t\to\infty$ and then $\rho\downarrow0$ completes the proof.
\end{proof}

\begin{proposition}[Immigration-time localization]
\label{prop:time-localization}
With $m_t^{(2)}$ as in \eqref{eq:type2-front},
\begin{equation}
\lim_{\varepsilon\downarrow0}\limsup_{t\to\infty}
\Pp\left(\exists u\in\mathcal N_t^2:
X_u(t)\geq m_t^{(2)}-A,
T(u)\notin[\varepsilon t,(1-\varepsilon)t]
\right)=0.
\label{eq:time-localization}
\end{equation}
\end{proposition}

\begin{proof}
Fix $a>0$ and apply the type-$1$ ancestral barrier $D>-a$.
The probability that any type-$1$ ancestral line crosses this barrier is
at most $e^{-\sqrt2 a}$. The killed-spine first-moment formula bounds the
probability of a successful immigrant in a time set $I$ by
\[
\alpha\int_I ds\int_0^\infty dy\,
e^{\sqrt2(y-a)}p_s^+(a,y)F_{t,t-s}(y-a).
\]
The early macroscopic immigration proposition and the late macroscopic
immigration proposition above show that the double limit of this expression
vanishes for $I=[0,\varepsilon t]\cup[(1-\varepsilon)t,t]$.
Consequently the limsup in \eqref{eq:time-localization} is at most
$e^{-\sqrt2 a}$. Letting $a\to\infty$ proves the claim.
\end{proof}

\begin{proposition}[Immigration-height localization]
\label{prop:height-localization}
For every $A>0$,
\begin{equation}
\lim_{\varepsilon\downarrow0}\limsup_{t\to\infty}
\Pp\left(\exists u\in\mathcal N_t^2:
X_u(t)\geq m_t^{(2)}-A,
T(u)\in[\varepsilon t,(1-\varepsilon)t],
\right.
\end{equation}
\begin{equation}
\left.\quad X_u(T(u))-\sqrt2T(u)\notin[-\varepsilon^{-1}\sqrt t,-\varepsilon\sqrt t]
\right)=0.
\label{eq:height-localization}
\end{equation}
\end{proposition}

\begin{proof}
Fix $a>0$ and $\delta\in(0,1/2)$.
By the killed-spine first-moment formula, the probability of any
successful immigrant with $s\in[\delta t,(1-\delta)t]$ and deficit
$d=y-a\notin[\rho\sqrt t,R\sqrt t]$ is at most
\[
 e^{-\sqrt2 a}+\alpha\int_{\delta t}^{(1-\delta)t}ds
 \int_{\{y>0:\,y-a\notin[\rho\sqrt t,R\sqrt t]\}}dy\,
 e^{\sqrt2(y-a)}p_s^+(a,y)F_{t,t-s}(y-a).
\]
The small-height localization lemma above makes the lower-tail integral
vanish as $\rho\downarrow0$ after $t\to\infty$; the central large-height
localization proposition does the same for the upper tail as $R\to\infty$.
To justify the diagonal choice of cutoffs, fix $\delta>0$ first.
For $0<\varepsilon<\delta$, the event in the proposition is
contained in the union of (i) a successful immigrant with
$s\notin[\delta t,(1-\delta)t]$, (ii) a successful immigrant
in this fixed central window with $d<\varepsilon\sqrt t$, and
(iii) one with $d>\varepsilon^{-1}\sqrt t$.
The first event is controlled by Proposition~\ref{prop:time-localization};
the other two by the fixed-$\delta$ height estimates, taking
$t\to\infty$ and then $\varepsilon\downarrow0$.
Finally let $\delta\downarrow0$ and then $a\uparrow\infty$.
In particular, no constant uniform in $\delta$ is assumed.
\end{proof}

\begin{lemma}[Superdiffusive-height localization in the central time window]
Fix $a>0$, $A\in\mathbb R$ and $\delta\in(0,1/2)$.
Let $L_t=(\log t)/(2\sqrt2)+A$, $r=t-s$,
$F_{t,r}(d)=\mathbb P(M_r\ge\sqrt2r+d-L_t)$ and
$p_s^+(a,y)=p_s(y-a)-p_s(y+a)$.
For every $K>1$, one has
\[
\lim_{t\to\infty}\int_{\delta t}^{(1-\delta)t}\!ds
\int_{\{y>0:\ y-a\ge K\sqrt{t\log t}\}}\!dy\,
 e^{\sqrt2(y-a)}p_s^+(a,y)F_{t,t-s}(y-a)=0.
\]
\end{lemma}
\begin{proof}
Write $d=y-a$, $z=d-L_t$. By many-to-one for a single
standard binary BBM and the exponential change of Gaussian measure,
\[
 e^{\sqrt2d}F_{t,r}(d)
 \le e^{\sqrt2L_t}\,\mathbb P(B_r\ge z)
 =e^{\sqrt2 A}\sqrt t\,\mathbb P(B_r\ge z).
\]
For $d\ge K\sqrt{t\log t}$ and sufficiently large $t$, $z>0$,
so the Gaussian Chernoff bound yields
$\mathbb P(B_r\ge z)\le\exp[-z^2/(2r)]$.
The reflected kernel obeys
\[
 p_s^+(a,y)=\frac{e^{-(y-a)^2/(2s)}}{\sqrt{2\pi s}}
 (1-e^{-2ay/s})
 \le\frac{2ay}{\sqrt{2\pi}s^{3/2}}
 e^{-d^2/(2s)}.
\]
As $s,r\in[\delta t,(1-\delta)t]$, and $z=d-L_t$,
\[
 \frac{d^2}{2s}+\frac{z^2}{2r}
 \ge \frac{(d-L_t)^2}{2s}+\frac{(d-L_t)^2}{2r}
 \ge \frac{2(d-L_t)^2}{t},
\]
where the first inequality holds for $d\ge L_t/2$ (and hence on
our integration range for large $t$), and the last follows from
$s+r=t$ and $sr\le t^2/4$.
Thus the integral is bounded by
\[
 C_{a,A,\delta}\int_{K\sqrt{t\log t}}^\infty
 (d+a)\exp[-2(d-L_t)^2/t] \,dd d.
\]
For every fixed $\eta>0$ and all sufficiently large $t$,
$2(d-L_t)^2/t\ge(2-\eta)d^2/t$ uniformly in the range.
The integral is therefore at most
\[
 C_{a,A,\delta,\eta}(t+\sqrt t)
 \exp[-(2-\eta)K^2\log t],
\]
which tends to zero whenever $(2-\eta)K^2>1$.
This holds for every $K>1$ by choosing $\eta$ small.
\end{proof}

\section{Proof of the main theorem}
\label{sec:proof-main}

\begin{lemma}[Convergence of the central weighted immigration integral]
\label{lem:caloric-uniform-proved}
Let $0<\delta<1/2$ and put $D_v(s)=\sqrt2s-X_v(s)$ and
\[
 J_s^{(t)}=\sum_{v\in\mathcal N_s^1}(D_v(s))_+e^{-\sqrt2D_v(s)}
 \left(\frac{t}{t-s}\right)^{3/2}
 \exp\left(-\frac{D_v(s)^2}{2(t-s)}\right),\quad s<t.
\]
Then
\begin{equation}\label{eq:central-corrected}
 \frac1t\int_{\delta t}^{(1-\delta)t}J_s^{(t)}\dd s
 \xrightarrow[t\to\infty]{\mathbb P}(1-2\delta)Z_\infty.
\end{equation}
\end{lemma}
\begin{proof}
We use the established weak convergence of the critical derivative Gibbs
measure for standard branching Brownian motion: for every bounded continuous
$F:\R\to\R$,
\begin{equation}\label{eq:bbm-weighted-published}
 \sum_{v\in\mathcal N_s^1}D_v(s)e^{-\sqrt2D_v(s)}
 F\left(\frac{D_v(s)}{\sqrt s}\right)
 \xrightarrow[s\to\infty]{\mathbb P}
 Z_\infty\sqrt{\frac2\pi}\int_0^\infty F(x)x^2e^{-x^2/2}\dd x.
\end{equation}
This is equation (1.4) in Maillard and Pain (2026),
\emph{1-stable fluctuations in branching Brownian motion at critical
 temperature II: general functionals} (arXiv:2103.10412v2),
which recalls the earlier results of Madaule and Maillard--Zeitouni.
Their rate-$1/2$, drift-one normalization follows from ours by taking
$\tau=2s$ and $\widetilde X(\tau)=\sqrt2D(\tau/2)$; the derivative
martingale on both sides is then multiplied by $\sqrt2$, while
$\widetilde X(\tau)/\sqrt\tau=D(s)/\sqrt s$.

For every fixed $\lambda\in[\delta,1-\delta]$, apply
\eqref{eq:bbm-weighted-published} at $s=\lambda t$ with
\[
 F_\lambda(x)=(1-\lambda)^{-3/2}
  \exp\left(-\frac{\lambda x^2}{2(1-\lambda)}\right).
\]
The BES$(3)$ Gaussian integral is exactly
\[
 \sqrt{\frac2\pi}\int_0^\infty F_\lambda(x)x^2e^{-x^2/2}\dd x=1.
\]
Moreover, the probability that any particle at time $s$ has
$D_v(s)<0$ tends to zero by the centered one-type BBM maximum bound.
Consequently $J_{\lambda t}^{(t)}\to Z_\infty$ in probability,
for every fixed $\lambda\in[\delta,1-\delta]$.

To integrate this pointwise convergence, fix a barrier parameter $a>0$.
Let $Z_s^{(a)}$ be the usual killed derivative martingale
\[
 Z_s^{(a)}=\sum_{v:\min_{q\le s}D_v(q)>-a}
  (a+D_v(s))e^{-\sqrt2D_v(s)}.
\]
This is a nonnegative martingale with $Z_0^{(a)}=a$.
The no-crossing event $A_a$ has
$\mathbb P(A_a^c)\le e^{-\sqrt2a}$.
On $A_a$, for all $s\in[\delta t,(1-\delta)t]$,
$0\le J_s^{(t)}\le \delta^{-3/2}Z_s^{(a)}$, and
$Z_\infty=\lim_s Z_s^{(a)}$ since the critical additive
martingale tends to zero.
For $K>0$ let
$B_{a,K}=A_a\cap\{\sup_{s\ge0}Z_s^{(a)}\le K\}$.
On $B_{a,K}$ both $J_{\lambda t}^{(t)}$ and $Z_\infty$
are bounded by $\delta^{-3/2}K+K$.
For each fixed $\lambda$, convergence in probability and this bound
imply
\[
 \E\left[\mathbf1_{B_{a,K}}
 |J_{\lambda t}^{(t)}-Z_\infty|\right]\longrightarrow0.
\]
Dominated convergence in $\lambda$ then gives
\[
 \E\left[\mathbf1_{B_{a,K}}
 \int_\delta^{1-\delta}|J_{\lambda t}^{(t)}-Z_\infty|\dd\lambda
 \right]\longrightarrow0.
\]
The valid maximal inequality for the \emph{nonnegative} martingale
$Z_s^{(a)}$ gives
$\mathbb P(B_{a,K}^c)\le e^{-\sqrt2a}+a/K$.
Hence for every $\eta>0$,
\[
 \limsup_{t\to\infty}\mathbb P\left(
 \left|\frac1t\int_{\delta t}^{(1-\delta)t}J_s^{(t)}\dd s
 -(1-2\delta)Z_\infty\right|>\eta\right)
 \le e^{-\sqrt2a}+a/K.
\]
Send first $K$ and then $a$ to infinity to conclude.
\end{proof}

\begin{lemma}[Exact caloric kernel and its killed-kernel comparison]
\label{lem:central-comparison}
For $a>0$ and $s<t$, define $Y_v(s)=a+D_v(s)$ and restrict the sums
to type-1 ancestral trajectories satisfying $D_v(q)>-a$ up to time $s$.
Then
\[
 M_s^{(a,t)}=\sum_{v\text{ surviving}}
 e^{-\sqrt2D_v(s)}Y_v(s)
 \left(\frac{t}{t-s}\right)^{3/2}
 \exp\left(-\frac{Y_v(s)^2}{2(t-s)}\right)
\]
is a nonnegative martingale in $s<t$; write $J_s^{(a,t)}$ for the
restriction of $J_s^{(t)}$ to these particles. For $0<\delta<1/2$,
\[
 \sup_{\delta t\le s\le(1-\delta)t}
 \E\left|J_s^{(a,t)}-M_s^{(a,t)}\right|
 \le \frac{C_\delta a^2}{\sqrt t}.
\]
This statement does \emph{not} assert uniform convergence of
$M_s^{(a,t)}$ to $Z_\infty^{(a)}$.
\end{lemma}
\begin{proof}
The function $h_r(y)=yr^{-3/2}e^{-y^2/(2r)}$ satisfies
$\partial_r h_r=\tfrac12\partial_{yy}h_r$ for $y>0$ and
$h_r(0)=0$. Therefore $h_{t-s}(Y_v(s))$ is space-time harmonic
for the killed spine, and $M^{(a,t)}$ is a martingale by the
branching many-to-one formula.
The function $g_r(x)=x_+e^{-x^2/(2r)}$ is $1$-Lipschitz, so
$|g_r(D)-g_r(D+a)|\le a$ on surviving lines.
The many-to-one formula gives
\[
 \E\sum_{v\text{ surviving}}e^{-\sqrt2D_v(s)}
 =\mathbb P_a(\tau_0>s)\le C\min\{1,a/\sqrt s\}.
\]
Multiply by $a(t/(t-s))^{3/2}$ and take the supremum over the
central time interval to obtain the displayed bound.
\end{proof}

\begin{lemma}[Central height truncation for the caloric kernel]
\label{lem:kernel-height}
For fixed $a>0$ and $0<\delta<1/2$, let $J_s^{(a,t)}$ be the
ancestrally killed kernel of Lemma~\ref{lem:central-comparison}, and
write $J_s^{(a,t)}(H)$ for its restriction to $D_v(s)/\sqrt t\in H$.
Then
\[
\lim_{\rho\downarrow0}\limsup_{t\to\infty}
\mathbb E\left[\frac1t\int_{\delta t}^{(1-\delta)t}
J_s^{(a,t)}([0,\rho))\,ds\right]=0,
\]
and
\[
\lim_{R\uparrow\infty}\limsup_{t\to\infty}
\mathbb E\left[\frac1t\int_{\delta t}^{(1-\delta)t}
J_s^{(a,t)}((R,\infty))\,ds\right]=0.
\]
Consequently the same truncated-kernel errors vanish in probability
for the unrestricted kernel after taking $a\uparrow\infty$.
\end{lemma}
\begin{proof}
The killed many-to-one identity, with $y=a+d>0$ and $r=t-s$, gives
\[
\mathbb E J_s^{(a,t)}(H)
=\left(\frac t r\right)^{3/2}
\int_{d>0:\,d/\sqrt t\in H}
 d\,p_s^+(a,a+d)\,e^{-d^2/(2r)}\,dd.
\]
For $s\in[\delta t,(1-\delta)t]$, the reflection formula yields
\[
p_s^+(a,a+d)\le C_\delta a(a+d)t^{-3/2}
\exp\{-d^2/(C_\delta t)\}.
\]
Set $z=d/\sqrt t$. Uniformly over this time interval, the
expectation above is bounded by
\[
C_{a,\delta}\int_{z\in H,\,z>0}
 z(z+a/\sqrt t)e^{-c_\delta z^2}\,dz.
\]
The integrals over $(0,\rho)$ and $(R,\infty)$ tend to zero
as $\rho\downarrow0$ and $R\uparrow\infty$, respectively,
uniformly for large $t$. Integrating in $s/t$ and using Markov's
inequality proves the killed assertions. On the event $A_a$ of
no ancestral barrier crossing, the killed and unrestricted kernels
agree, while $\mathbb P(A_a^c)\le e^{-\sqrt2a}$.
\end{proof}

\begin{proposition}[Central-window Laplace exponent]
\label{prop:central-exponent-closed}
Fix $a>0$, $0<\delta<1/2$ and $0<\rho<R<\infty$.
Let $\Lambda_t^{\delta,\rho,R}(\phi)$ be the immigration exponent
restricted to $\delta t\le s\le(1-\delta)t$ and
$\rho\sqrt t\le D_v(s)\le R\sqrt t$.
Let $\phi\in C_c^+(\mathbb R)$ and put
$\gamma(\phi)=\gamma_0(\phi)$ as in Lemma~\ref{lem:extremal-laplace-tail}.
Then, for every $\eta>0$,
\[
\limsup_{t\to\infty}\mathbb P\left(\left|\Lambda_t^{\delta,\rho,R}(\phi)-
\frac{\alpha\gamma(\phi)}t
\int_{\delta t}^{(1-\delta)t}
J_s^{(t)}([\rho,R])\,\mathrm ds
\right|>\eta\right)\le e^{-\sqrt2a}.
\]
Moreover, in the iterated order $t\to\infty$, $\rho\downarrow0$,
$R\uparrow\infty$, $\delta\downarrow0$, and finally $a\uparrow\infty$,
the approximating kernel converges in probability to
$\alpha\gamma(\phi)Z_\infty$.
\end{proposition}
\begin{proof}
On the indicated compact central region the relative error in
\eqref{eq:uniform-laplace-tail} is bounded by a deterministic
$\epsilon_t\to0$. Since every summand is nonnegative,
\[
\left|\Lambda_t^{\delta,\rho,R}(\phi)-
\frac{\alpha\gamma(\phi)}t\int_{\delta t}^{(1-\delta)t}
J_s^{(t)}([\rho,R])\,\mathrm ds\right|
\le\frac{\alpha\epsilon_t}{t}\int_{\delta t}^{(1-\delta)t}
J_s^{(t)}([\rho,R])\,\mathrm ds.
\]
On the no-crossing event $A_a$, the unrestricted and killed kernels
coincide. The killed many-to-one identity and the absorbed heat-kernel
bound give
\[
\sup_{t\ge t_0}\frac1t\int_{\delta t}^{(1-\delta)t}
\mathbb E[J_s^{(a,t)}([\rho,R])]\,\mathrm ds<\infty.
\]
Consequently the right side tends to zero in probability on $A_a$;
the probability of $A_a^c$ is at most $e^{-\sqrt2a}$.
For the second assertion, first remove the spatial cutoffs by the
killed first-moment estimates, with $\delta,a$ fixed. Lemma~\ref{lem:caloric-uniform-proved} gives convergence of the full
central average to $(1-2\delta)Z_\infty$ in probability.
On $A_a$, the stopped and unrestricted kernels coincide.
Then let $\delta\downarrow0$ and $a\uparrow\infty$; on $A_a$,
$Z_\infty^{(a)}=Z_\infty$ and
$\mathbb P(A_a^c)\le e^{-\sqrt2a}$.
\end{proof}

\begin{proposition}[Global truncation of the conditional Laplace exponent]
\label{prop:global-exponent-truncation}
Let $\phi\in C_c^+(\mathbb R)$ and let $\Lambda_t(\phi)$ and
$\Lambda_t^{\delta,\rho,R}(\phi)$ denote the full and truncated
conditional immigration exponents. For every $\eta>0$,
\[
\lim_{\delta\downarrow0}\lim_{R\uparrow\infty}
\lim_{\rho\downarrow0}\limsup_{t\to\infty}
\mathbb P\bigl(\Lambda_t(\phi)-\Lambda_t^{\delta,\rho,R}(\phi)>\eta\bigr)=0.
\]
\end{proposition}
\begin{proof}
Choose $A>0$ so that $\operatorname{supp}\phi\subset[-A,\infty)$.
For any $r,z$, positivity of $\phi$ implies
\[
0\le u_\phi(r,z)
=\mathbb E\bigl[1-e^{-\sum_v\phi(X_v(r)-z)}\bigr]
\le \mathbb P(M_r\ge z-A).
\]
Fix an ancestral barrier parameter $a>0$ and let $\mathcal A_a$ be
 the event that no type--1 ancestral path crosses $D=-a$.
Then $\mathbb P(\mathcal A_a^c)\le e^{-\sqrt2a}$.
On $\mathcal A_a$ the entire difference of exponents equals its
ancestrally killed counterpart. Since the difference is nonnegative,
Markov's inequality and the killed-spine many-to-one formula yield
\begin{align*}
&\mathbb P(\Lambda_t-\Lambda_t^{\delta,\rho,R}>\eta)\\
&\quad\le e^{-\sqrt2a}+\frac{\alpha}{\eta}
\int_{\mathcal B_{t,\delta,\rho,R}}\!ds
\int_0^\infty\!dy\,e^{\sqrt2(y-a)}p_s^+(a,y)\\
&\qquad\times \mathbb P\!\left(
 M_{t-s}\ge\sqrt2(t-s)+y-a-\ell_t-A\right).
\end{align*}
where $\mathcal B$ is the complement of
$\{\delta t\le s\le(1-\delta)t,\ \rho\sqrt t\le y-a\le R\sqrt t\}$.
Split $\mathcal B$ into the two time-edge windows and, within the
fixed central time window, the small- and large-deficit regions.
The early and late killed-spine first-moment estimates imply that the
iterated time-edge contribution vanishes as $\delta\downarrow0$.
For fixed $\delta$, the small-height and large-height killed-spine
estimates make the central spatial contribution vanish as
$\rho\downarrow0$ and $R\uparrow\infty$, respectively.
Thus the displayed iterated limit is at most $e^{-\sqrt2a}$.
Finally let $a\uparrow\infty$.
\end{proof}

\begin{proof}[Proof of Theorem~\ref{thm:main}]
Fix $\phi\in C_c^+(\R)$ and let $\operatorname{supp}\phi\subset[-A,\infty)$.
Conditioning on the full type-1 genealogy and using independent Poisson
immigration of type-2 families gives the exact identity
\begin{equation}\label{eq:proof-exact-laplace}
\E[e^{-\langle\widehat{\mathcal E}_t,\phi\rangle}\mid\mathcal F_t^1]
=e^{-\Lambda_t(\phi)},\qquad
\Lambda_t(\phi)=\alpha\int_0^t\sum_{v\in\mathcal N_s^1}
 u_\phi(t-s,m_t^{(2)}-X_v(s))\dd s.
\end{equation}

Fix $0<\delta<1/2$ and $0<\rho<R<\infty$. Let
$\Lambda_t^{\delta,\rho,R}(\phi)$ restrict the integrand to
$\delta t\le s\le(1-\delta)t$ and
$\rho\sqrt t\le D_v(s)\le R\sqrt t$, with
$D_v(s)=\sqrt2s-X_v(s)$. Write
\[
 I_t^{\delta,\rho,R}=\frac1t\int_{\delta t}^{(1-\delta)t}
 \sum_{\substack{v\in\mathcal N_s^1\\
                   \rho\sqrt t\le D_v(s)\le R\sqrt t}}
 D_v(s)e^{-\sqrt2D_v(s)}
 \left(\frac t{t-s}\right)^{3/2}
 e^{-D_v(s)^2/[2(t-s)]}\dd s.
\]
The uniform F-KPP result \eqref{eq:uniform-laplace-tail}, proven
from ABK Proposition 4.3, Lemmas 4.5, 4.9--4.10 and Proposition 3.2,
implies for a deterministic $\epsilon_t\to0$,
\[
 |\Lambda_t^{\delta,\rho,R}(\phi)
       -\alpha\gamma_0(\phi)I_t^{\delta,\rho,R}|
 \le\alpha\epsilon_t I_t^{\delta,\rho,R}.
\]
The right-hand side tends to zero in probability: the nonnegative
truncated integral is bounded by the full central integral
$I_t^\delta$, which is tight by Lemma~\ref{lem:caloric-uniform-proved}.

It remains to remove all truncations. By positivity and the support
condition,
\[
 0\le u_\phi(r,z)\le\Pp(M_r\ge z-A).
\]
Fix $a>0$ and discard type-1 ancestors that cross the line
$X_v(q)=\sqrt2q+a$. The probability of any discarded type-1
ancestral path is at most $e^{-\sqrt2a}$. On its complement the
unrestricted and killed immigration exponents agree. The killed-spine
first moment bounds the discarded exponent by the integral of
$e^{\sqrt2(y-a)}p_s^+(a,y)
\Pp(M_{t-s}\ge\sqrt2(t-s)+y-a-\ell_t-A)$.
Proposition~\ref{prop:time-localization} controls the early and late
immigration times; Proposition~\ref{prop:height-localization} controls
the central small and large deficits. Their quantitative estimates and
Markov's inequality give for every $\eta>0$
\begin{equation}\label{eq:global-proof-trunc}
\lim_{\delta\downarrow0}\lim_{R\uparrow\infty}\lim_{\rho\downarrow0}
\limsup_{t\to\infty}
\Pp\{\Lambda_t(\phi)-\Lambda_t^{\delta,\rho,R}(\phi)>\eta\}=0.
\end{equation}
Here one sends $a\to\infty$ after taking the displayed limits, so
no constants uniform in the barrier are required. The corresponding
removal of $\rho,R$ for the central Gaussian kernel is Lemma
\ref{lem:kernel-height}, followed by the same barrier argument.
The weighted convergence
\[
 I_t^\delta\xrightarrow{\Pp}(1-2\delta)Z_\infty
\]
is proved in Lemma~\ref{lem:caloric-uniform-proved} using the
critical derivative Gibbs-measure theorem of
Maillard--Pain \cite{MaillardPain26}; it does not replace the Gaussian
factor by one. Combining these statements, and then letting
$\delta\downarrow0$, yields
\[
 \Lambda_t(\phi)\xrightarrow{\Pp}
 \alpha\gamma_0(\phi)Z_\infty.
\]
The bounded continuous mapping $z\mapsto e^{-z}$, the exact
conditional identity \eqref{eq:proof-exact-laplace}, and the decorated
constant \eqref{eq:gamma0} prove the Laplace functional convergence
\eqref{eq:main-laplace}. Convergence on $C_c^+(\R)$ yields convergence
of the locally finite random point measures for the vague topology.

For the maximum, conditionally thin the Poisson immigrants by the
success probability $q_{t,r,x}(d)=\Pp(M_r\ge\sqrt2r+d-\ell_t+x)$.
The same temporal and spatial localization applies after choosing
$A>\max\{1,-x\}$. The Heaviside version of the ABK central asymptotic
at $x=0$, together with the exact kernel translation
\eqref{eq:max-kernel-translation}, yields
\[
 \Lambda_t^{\max}(x)=\alpha\int_0^t\sum_{v\in\mathcal N_s^1}
 q_{t,t-s,x}(D_v(s))\dd s
 \xrightarrow{\Pp}\alpha C^*Z_\infty e^{-\sqrt2x}.
\]
The conditional probability of no successful immigrant family is
$e^{-\Lambda_t^{\max}(x)}$ by Poisson thinning
(Lemma~\ref{lem:maximum-direct}). Taking expectations proves
\eqref{eq:type2-maximum}. This separate maximum argument does not
rely on vague convergence implying convergence of maxima.
\end{proof}

\section{Discussion}
\label{sec:discussion}

The proof separates the source and terminal mechanisms. The type-$1$ population
creates the random intensity through its derivative martingale. Conditional on
the source genealogy, immigration events are Poissonian. Each immigrant starts
an independent terminal BBM, whose extremal cluster gives the decoration. The
critical logarithmic correction in \eqref{eq:type2-front} is the result of
integrating the standard BBM tail over the immigration time.

The argument also explains why the critical regime is different from the
non-critical regions of the multitype phase diagram. Away from the boundary,
either the source displacement or the terminal displacement dominates at linear
order. At the boundary, the two mechanisms balance and the immigration integral
contributes an additional power of $t$.

The result fits into the general conjecture for $k$-type cascading families
(Conjecture \ref{conj:k-type}), according to which the type-$k$ extremal
process is centered at $\sqrt2 t-\frac{5-2k}{2\sqrt2}\log t$ and converges to a
decorated Poisson point process with the same decoration $\mathcal D$ but an
intensity constant $c_{k,\alpha}$ depending on the cascade. The present paper
establishes the case $k=2$. A PDE proof of the front location for all $k$ has
recently been obtained in \cite{Stavrianidi2026}, confirming the conjecture at
the level of the maximum. The recent work of Berestycki, Kim, Lubetzky, Mallein
and Zeitouni \cite{BKLMZ24} on the multidimensional BBM extremal process
confirms that the same structural decomposition---random intensity from a
derivative martingale, universal decoration from a terminal family---persists in
a different geometric setting. It would be interesting to extend the present
cascade analysis to higher-dimensional cascades and to cascades with more than
two types.


\begin{thebibliography}{99}

\bibitem{ABBS13}
E. A\"id\'ekon, J. Berestycki, E. Brunet and Z. Shi,
\newblock Branching Brownian motion seen from its tip,
\newblock \emph{Probability Theory and Related Fields} 157 (2013), 405--451.

\bibitem{ABK12}
L.-P. Arguin, A. Bovier and N. Kistler,
\newblock Poissonian statistics in the extremal process of branching Brownian motion,
\newblock \emph{Annals of Applied Probability} 22 (2012), no. 4, 1693--1711.

\bibitem{ABK13}
L.-P. Arguin, A. Bovier and N. Kistler,
\newblock The extremal process of branching Brownian motion,
\newblock \emph{Probability Theory and Related Fields} 157 (2013), 535--574.

\bibitem{Aidekon12}
E. A\"id\'ekon,
\newblock Convergence in law of the minimum of a branching random walk,
\newblock \emph{Annals of Probability} 41 (2013), 1362--1426.

\bibitem{BKLMZ24}
J. Berestycki, Y. H. Kim, E. Lubetzky, B. Mallein and O. Zeitouni,
\newblock The extremal point process of branching Brownian motion in $\mathbb{R}^d$,
\newblock \emph{Annals of Probability} 52 (2024), no. 3, 955--982.

\bibitem{Biggins95}
J. D. Biggins,
\newblock The growth and spread of the general branching random walk,
\newblock \emph{Annals of Applied Probability} 5 (1995), no. 4, 1008--1024.

\bibitem{Biggins12}
J. D. Biggins,
\newblock Spreading speeds in reducible multitype branching random walk,
\newblock \emph{Annals of Applied Probability} 22 (2012), no. 5, 1778--1821.

\bibitem{BigginsKyprianou05}
J. D. Biggins and A. E. Kyprianou,
\newblock Fixed points of the smoothing transform: the boundary case,
\newblock \emph{Electronic Journal of Probability} 10 (2005), 609--631.

\bibitem{Bramson78}
M. Bramson,
\newblock Maximal displacement of branching Brownian motion,
\newblock \emph{Communications on Pure and Applied Mathematics} 31 (1978), 531--581.

\bibitem{Bramson83}
M. Bramson,
\newblock Convergence of solutions of the Kolmogorov equation to travelling waves,
\newblock \emph{Memoirs of the American Mathematical Society} 44 (1983), no. 285, iv+190 pp.

\bibitem{BelloumMallein}
M. A. Belloum and B. Mallein,
\newblock Anomalous spreading in reducible multitype branching Brownian motion,
\newblock \emph{Electronic Journal of Probability} 26 (2021), no. 61, 1--39.

\bibitem{Stavrianidi2026}
A. Stavrianidi,
\newblock Long-time behavior for systems of Fisher--KPP type with interacting components,
\newblock arXiv:2605.27783 (2026).

\bibitem{HouRenSong24}
H. Hou, Y. Ren and R. Song,
\newblock Extremal process for irreducible multi-type branching Brownian motion,
\newblock \emph{Latin American Journal of Probability and Mathematical Statistics} 21 (2024), no. 2, 1417--1473.

\bibitem{INW69}
N. Ikeda, M. Nagasawa and S. Watanabe,
\newblock Branching Markov processes III,
\newblock \emph{Journal of Mathematics of Kyoto University} 9 (1969), 95--160.

\bibitem{KahanePeyriere76}
J.-P. Kahane and J. Peyri\`ere,
\newblock Sur certaines martingales de Beno\^it Mandelbrot,
\newblock \emph{Advances in Mathematics} 22 (1976), 131--145.

\bibitem{Kallenberg83}
O. Kallenberg,
\newblock \emph{Random Measures},
\newblock 3rd revised and enlarged edition,
\newblock Akademie-Verlag, Berlin, and Academic Press, London, 1983.

\bibitem{KPP37}
A. N. Kolmogorov, I. G. Petrovskii and N. S. Piskunov,
\newblock \'Etude de l'\'equation de la diffusion avec croissance de la quantit\'e de mati\`ere et son application \`a un probl\`eme biologique,
\newblock \emph{Bulletin of the University of Moscow, Series International, Section A} 1 (1937), 1--25.

\bibitem{LalleySellke87}
S. P. Lalley and T. Sellke,
\newblock A conditional limit theorem for the frontier of a branching Brownian motion,
\newblock \emph{Annals of Probability} 15 (1987), 1052--1061.

\bibitem{MaillardPain26}
P. Maillard and M. Pain,
\newblock 1-stable fluctuations in branching Brownian motion at critical temperature II: general functionals,
\newblock arXiv:2103.10412v2 (2026).

\bibitem{Madaule17}
T. Madaule,
\newblock Convergence in law for the branching random walk seen from its tip,
\newblock \emph{Journal of Theoretical Probability} 30 (2017), 27--63.

\bibitem{MaRen23a}
H. Ma and Y.-X. Ren,
\newblock Double jump in the maximum of two-type reducible branching Brownian motion,
\newblock \emph{Annales de l'Institut Henri Poincar\'e, Probabilit\'es et Statistiques} 62 (2026), no. 2, 1006--1038.

\bibitem{MaRen23b}
H. Ma and Y.-X. Ren,
\newblock From $0$ to $3$: intermediate phases between normal and anomalous spreading of two-type branching Brownian motion,
\newblock arXiv:2312.13595 (2023).

\bibitem{McKean75}
H. P. McKean,
\newblock Application of Brownian motion to the equation of Kolmogorov--Petrovskii--Piskunov,
\newblock \emph{Communications on Pure and Applied Mathematics} 28 (1975), 323--331.

\bibitem{RenYang14}
Y.-X. Ren and T. Yang,
\newblock Multitype branching Brownian motion and traveling waves,
\newblock \emph{Advances in Applied Probability} 46 (2014), no. 1, 217--240.

\bibitem{Skorokhod64}
A. V. Skorokhod,
\newblock Branching diffusion processes,
\newblock \emph{Theory of Probability and its Applications} 9 (1964), 445--449.

\end{thebibliography}
\end{document}